\documentclass[a4paper,12pt]{article}
\usepackage{amssymb,amsmath,amsfonts,amsthm}
\usepackage{graphicx,afterpage}
\usepackage{array} 
 
\newcommand{\C}{\ensuremath{\mathcal{C}}}

\newcommand{\diam}{\ensuremath{\mathrm{Diam} \; }}

 \newcommand{\graph}{\ensuremath{\C(G,X)}}

\newcommand{\zz}{\mathbb{Z}}

\newcommand{\rr}{\mathbb{R}}

\newcommand{\bu}{\mathbf{u}}
\newcommand{\bv}{\mathbf{v}}
\newcommand{\bw}{\mathbf{w}}
\newcommand{\bo}{\mathbf{0}}
\newtheorem{thm}{Theorem}[section]
\newtheorem{lemma}[thm]{Lemma}
\newtheorem{prop}[thm]{Proposition}

\theoremstyle{definition}

\newtheorem{defn}[thm]{Definition}

\usepackage{graphicx}
\newcommand{\coverpage}[3]{\thispagestyle{empty}
	\addtocounter{page}{-1}
	\null\vspace*{-1cm} \hfill\includegraphics[scale=1]{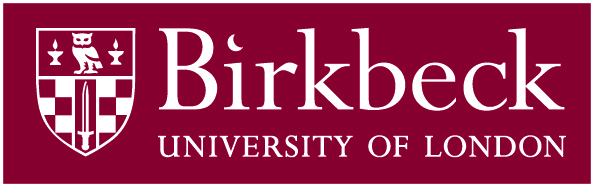} \vskip 2in
	\begin{center} \begin{minipage}{0.7\textwidth}\begin{center}\Huge\bf{#1}\end{center} \end{minipage}\end{center}  \vfill
	\begin{center} {\large By}\bigskip\\ {\large #2}\\ \end{center} \vfill 
	\framebox{\begin{minipage}{\textwidth}
			Birkbeck Mathematical Sciences Preprint Series\hfill
			Preprint Number #3 \\ \\
			\null\hfill www.bbk.ac.uk/ems/research/pure/preprints 
	\end{minipage}}
	\newpage}
\begin{document}
\coverpage{A Note on Commuting Involution Graphs in Affine Coxeter Groups}{Sarah Hart and Amal Sbeiti Clarke}{43}

\title{A Note on Commuting Involution Graphs in Affine Coxeter Groups}
\author{Sarah Hart and Amal Sbeiti Clarke}
\date{}
\maketitle


 \section{Introduction}

Let $G$ be a group and $X$ a set of involutions of $G$. The {\em commuting involution graph} $\mathcal{C}(G,X)$ is the graph whose vertex set is $X$, with vertices $x, y$ joined by an edge whenever $x$ and $y$ commute.
These graphs have been studied for a wide variety of groups, usually with $X$ being either a conjugacy class or the set of all involutions. Perhaps the most well-known example of their use was in the work of Fischer on 3-transposition groups \cite{fischer}. In a series of papers, Bates et al looked at connectedness and diameters of commuting involution graphs in the case where $X$ is a conjugacy class of involutions and $G$ is the symmetric group \cite{symmetric}, a finite Coxeter group \cite{Finite}, a linear group \cite{linear}, or a sporadic simple group \cite{sporadic}. In particular, for finite Coxeter groups they gave necessary and sufficient conditions under which the commuting involution graphs are connected, along with bounds for the diameters in the connected cases. For the symmetric group, where \graph\ is connected the diameter is at most 4. For finite Coxeter groups this bound increases to 5. In \cite{Perkins}, affine Coxeter groups of type $\tilde A_n$ were considered. Here, the diameter of any connected commuting involution graph is at most 6.\\

More recently, Hart and Sbeiti Clarke considered the remaining classical Weyl groups (see \cite{classicalaffine}, \cite{amal} and \cite{amalthesis}). They showed that if $G$ is a classical affine Weyl group, then if $\graph$ is connected, its diameter exceeds the rank of $G$ by at most 1. The obvious next question is: what happens in the exceptional affine groups? The purpose of this short note is to establish some general results for commuting involution graphs in affine Coxeter groups, and to deal with types $\tilde F_4$ and $\tilde G_2$. Types $\tilde E_6$ $\tilde E_7$ and $\tilde E_8$ are more substantial and these will be addressed in a forthcoming paper.\\

Section 2 contains preliminaries and general results. Section 3 deals with $\tilde F_4$ and $\tilde G_2$.

\section{Preliminaries and General Results}

Let $W$ be a finite Weyl group with root system $\Phi$ in a Euclidean vector space $V \cong \rr^n$, and $\Phi^{\vee}$ the set of coroots $\alpha^{\vee} = \frac{2\alpha}{\langle \alpha, \alpha\rangle}$, for $\alpha \in \Phi$. The affine Weyl group $\tilde W$ is the semidirect product of $W$ with translation group $Z$  of the coroot lattice $L(\Phi^\vee)$ of $W$. We often express roots and coroots in terms of the standard basis $\{e_1, \ldots, e_n\}$, and we will, by a slight abuse of notation, write elements of $Z$ simply as vectors -- in other words we will identify $Z$ with $L(\Phi^\vee)$. For any  $w \in \tilde W$, $w$ is written in the form $(a,\bu)$ where $a \in W$ and $\bu\in Z$. See, for example, \cite[Chapter 4]{humphreys} for more detail.

 For $a, b \in W$ and $\bu, \bv \in Z$   we have $$(a,\bu)(b,\bv) = (ab,\bu^b + \bv).$$
We have $(a,\bu)^{-1} =(a^{-1}, -\bu^{a^{-1}})$.
In $\tilde W$, 
the element $(a,\bu)$ is conjugate to $(b,\bv)$ via some $(g,\mathbf{w})$ if and only if \begin{align*} (b,\bv) &= (a, \bu)^{(g,\bw)}
= (g^{-1}a g, \bu^g + \bw - \bw^{g^{-1} a g}).
\end{align*}
   
The reflections of $\tilde W$ are the affine reflections $s_{{\alpha},k}$ ($\alpha \in\Phi$, $k \in\zz$). Recall that, for $\bv$ in $V$,  \begin{align*}
s_{{\alpha}}(\bv) &= \bv - 2\frac{\langle \alpha, \bv\rangle}{\langle \alpha, \alpha\rangle}\alpha = \bv - \langle \alpha, \bv\rangle \alpha^\vee\\
s_{{\alpha},k}(\bv) &= \bv - 2\frac{(\langle \alpha, \bv\rangle - k)}{\langle \alpha, \alpha\rangle}\alpha = s_{\alpha}(\bv) + k\alpha^{\vee}.
\end{align*}If $R$ is a set of  simple reflections for $W$ and $\tilde \alpha$ is the highest root (that is, the root with the highest coefficient sum when expressed as a linear combination of simple roots), then it can be shown that $R\cup\{s_{\tilde\alpha ,1}\}$ is a set of simple reflections for $\tilde W$.\\

 Finally, we   write $\diam \mathcal{C}(G,X)$ for the diameter of $\mathcal{C}(G,X)$
           when $\mathcal{C}(G,X)$ is a connected graph, in other words the maximum distance $d(x,y)$ between any $x, y \in X$ in the graph. \\

   We begin with an observation about  connectedness. For an element $x=( a,\bu)$ in a conjugacy class $X$ of $\tilde W$, we define $ \hat x= a$. Then let $ \hat X$ be the conjugacy class of $\hat x$ in $W$. Clearly if $x, y  \in X$, then $\hat x,\hat y \in \hat X$.     
           
           \begin{lemma} \label{obvious} Suppose $x,y \in X$. If $d(\hat x,
           	\hat y) = k$, then $d(x,y) \geq k$. If $\C (W,\hat X)$  is disconnected,
           	then $\C (\tilde W, X)$ is disconnected.
           \end{lemma}

           \begin{proof} The result follows immediately from the observation that if $x$ commutes with $y$ in $\tilde W$, then $\hat x$ commutes with $\hat y$ in $W$. 
           \end{proof}

   \begin{defn} \label{equiv}
   	Let $W$ be an arbitrary Coxeter group, with $R$ the set of simple reflections. Two subsets $I$ and $J$ of $R$ are {\em $W$-equivalent}
   	if there exists $w \in W$ such that $I^w = J$.
   \end{defn} 
   
   In the next result, we use the notation $w_{I}$ for the longest element of a finite standard parabolic subgroup $W_I$. 
   
   \begin{thm}[Richardson \cite{richardson}]\label{richardson}
   	
   	Let $W$ be an arbitrary Coxeter group, with $R$ the set of simple reflections. Let $g \in W$ be an
   	involution. Then there exists $I \subseteq R$ such that $w_I$ is central in $W_I$, and $g$ is conjugate to
   	$w_I$. In addition, for $I, J \subseteq R$, $w_I$ is conjugate to $w_J$ if and only if $I$ and $J$ are
   	$W$-equivalent.
   \end{thm}
  
  For the rest of this paper, $W$ will always denote a finite Weyl group with root system $\Phi$, with $\tilde W$ its corresponding affine Weyl group. We will use the convention that a Weyl group of type $\Gamma$ will be denoted $W(\Gamma)$, where $\Gamma$ is the associated Coxeter graph. 
   The next two lemmas give conditions under which reflections  
   $w$ and  $w'$ of $\tilde W$ commute.

   \begin {lemma}\label{commute1} Let $\alpha ,\beta \in \Phi$. Then $s_{\alpha}$ commutes with $s_{\beta}$ if and only if $\langle\alpha,\beta\rangle=0$ or $\alpha =\pm \beta$.
   \end{lemma}
   \paragraph{Proof} Consider $\bv \in V$ and $ \alpha, \beta \in \Phi$. A quick calculation shows that
   \begin{align*}
s_{\beta}s_{\alpha}(\bv)&= \bv -   \frac{2\langle \bv, \alpha\rangle}{\langle\alpha,\alpha\rangle}\alpha - \frac{2\langle \bv, \beta\rangle}{\langle\beta,\beta\rangle}\beta +   
   \frac{4\langle \bv, \alpha\rangle \langle \alpha,\beta\rangle }{\langle\alpha,\alpha\rangle \langle \beta,\beta \rangle}\beta;\\
   s_{\alpha}s_{\beta}(\bv)&= \bv- \frac{2\langle \bv, \alpha\rangle}{\langle\alpha,\alpha\rangle}\alpha -   \frac{2\langle \bv, \beta\rangle}{\langle\beta,\beta\rangle}\beta +   
   \frac{4\langle \bv, \beta\rangle \langle \alpha,\beta\rangle }{\langle\beta,\beta\rangle \langle \alpha,\alpha \rangle}\alpha. 
   \end{align*}
   
   Hence $s_{\alpha } s_{\beta} =s_{\beta } s_\alpha$  if and only if either $\langle \alpha, \beta \rangle = 0$ or 
   $\langle \bv, \beta\rangle \alpha =\langle \bv, \alpha\rangle\beta$ for all $\bv \in V$. One of the properties of root systems is that for any root $\gamma$ we have $\langle \gamma \rangle \cap \Phi = \{\pm \gamma\}$. Therefore $s_{\alpha } s_{\beta} =s_{\beta } s_\alpha$ if and only if $\langle \alpha,\beta \rangle =0$ or $\alpha =\pm \beta$.   \qed
   \begin{lemma} \label{commute}
   For all 
  positive roots $\alpha$ and $\beta$, and all integers $k$ and $l$, the affine reflections  $s_{\alpha,k}$ and $s_{\beta,l}$ commute if and only if either $\langle\alpha,\beta\rangle=0$, or $\alpha=\beta$ and $k=l$.
   \end{lemma}
   \begin{proof}
   Let $\alpha$ and $\beta$ be positive roots, with $k$ and $l$ integers. Then
   \begin{align*} 
   s_{\alpha,k} s_{\beta,l}(\bv )&= s_{\alpha,k}(s_ \beta (\bv) +l\beta ^\vee) = s_\alpha s_\beta(\bv)+ s_\alpha(l\beta^\vee)+k\alpha^\vee \\
  &=s_{\alpha}s_{\beta}(\bv) + k\alpha^\vee + l\beta^{\vee} - l\langle\alpha, \beta^\vee\rangle\alpha^\vee;\\
   s_{\beta,l} s_{\alpha,k}(\bv )&= s_\beta s_\alpha(\bv)+ k\alpha^\vee + l\beta^{\vee} - k\langle\alpha^\vee, \beta\rangle\beta^\vee.
   \end{align*}
   Consequently $s_{\alpha,k} s_{\beta,l}=s_{\beta,l} s_{\alpha,k}$ precisely when, for all $\bv \in V$, we have
   \begin{align*}
   s_\alpha s_\beta(\bv)-l\langle\alpha,\beta^\vee \rangle \alpha^\vee&= s_\beta s_\alpha(\bv)-k\langle\alpha^\vee,\beta \rangle \beta^\vee.
   \end{align*}
   In particular, setting $\bv=0$ we get
   $l\langle\alpha,\beta^\vee \rangle \alpha^\vee= k\langle\alpha^\vee,\beta \rangle \beta^\vee,$ which implies that either $\alpha=\beta$ or $\langle\alpha,\beta\rangle=0$. If $\alpha=\beta$,  we get $s_{\alpha,k}s_{\alpha,l}(\bv)= \bv +(l-k)\alpha^\vee, $ whereas $s_{\alpha,l}s_{\alpha,k}(\bv)= \bv + (k-l)\alpha^\vee$. Therefore, $s_{\alpha,k}$
   commutes with $s_{\alpha,l}$ if and only if $k=l$. On the other hand, if $\alpha \neq \beta$, then $\langle \alpha,\beta\rangle=0$ and so $s_\alpha s_\beta=s_\beta s_\alpha$. Thus $s_{\alpha,k} s_{\beta,l}=s_{\beta,l} s_{\alpha,k}$. \end{proof}

  There is one case we can deal with that occurs in several groups.
  
  \begin{lemma}
  	\label{centralinv} Suppose $W$ is a finite Weyl group that has a central involution $\tilde w$. Now suppose $x$ is an involution in the corresponding affine Weyl group $\tilde W$ such that $x = (\tilde w, \bu)$ for some $\bu \in Z$, and write $X = x^{\tilde W}$, the conjugacy class in $\tilde W$ of $x$. Then $\C(\tilde W,X)$ is disconnected.
  \end{lemma}
  
  \begin{proof}
  	Since $\tilde w$ is central in $W$, every conjugate of $x$ in $\tilde W$ has the form $y=(\tilde w, \bv)$ for some $\bu$ in $Z$. Moreover $\tilde w$ acts as $-1$ on the root system, and hence on $Z$. Now $$xy = (\tilde w,\bu)(\tilde w,\bv) = (\tilde w^2, \bu^{\tilde w} + \bv) = (1, \bv-\bu).$$
  	Thus $xy$ is translation through $\bv - \bu$. Now $x$ commutes with $y$ precisely when $(xy)^2 = 1$, and hence $x$ commutes with $y$ if and only if $\bu = \bv$, meaning $x$ does not commute with any other member of its conjugacy class. Therefore $\C(\tilde W, X)$ is completely disconnected.  
  \end{proof}

   Next we have a result which proves connectedness in certain circumstances.
	We write $x \leftrightarrow y$ to mean $xy = yx$. 
  
  \begin{prop}\label{alternative} Let $X$ be a conjugacy class of involutions in $\tilde W$ that contains $(a, \bo)$ for some involution $a$ of $W$. 
  Suppose that $\C(W,\hat X)$ is connected with diameter $d$, and further that there is an integer $k$ such that whenever $(a, \bu) \in X$, there is some $\hat x \in \hat X$ such that  $d((a, \bu),(\hat x, \bo)) \leq k$ in $\C(\tilde W, X)$. Then $\C(\tilde W, X)$ is connected with diameter at most $d+k$. 	
  \end{prop}
  
  \begin{proof} Let $(b,\bv) \in X$. Then $b$ is conjugate to $a$ in $W$. That is, there is some $g \in W$ with $b = a^g$, meaning that $(b, \bv)^{g^{-1}}$ is of the form $(a, \bu)$ for an appropriate $\bu$. By hypothesis then, there is a path $(a, \bu) = x_0 \leftrightarrow x_1 \cdots \leftrightarrow x_m = (\hat x,\bo)$ of elements $x_i$ of $X$, with $m \leq k$. But this implies that there is a path $x_0^g \leftrightarrow x_1^g \leftrightarrow \cdots \leftrightarrow x_m^g$ in the commuting involution graph.  Writing $y = \hat x^g$, and noting that $x_0^g = (b,\bv)$, we obtain $d((b,\bv), (y, \bo)) \leq k$. Now $\C(W,\hat X)$ is connected with diameter $d$. Thus $d(y, a) \leq d$ in $\C(W,\hat X)$. Hence $d((y,\bo),(a, \bo)) \leq d$ in $\C(\tilde W, X)$. Therefore $d((b,\bv),(a,\bo))\leq d+k$ and since this holds for all $(b,\bv)$ in $X$, we deduce that $\C(\tilde W, X)$ is connected with diameter at most $d+k$. \end{proof}
  
  \begin{lemma}\label{lem28sep}
  	Suppose $a \in W_I$ for some $I \subseteq R$. If $(a,\bu)$ is an involution, with $\bu = \sum_{r \in R} u_r \alpha_r$, then $a$ is an involution and $u_r = 0$ whenever $r \notin I$. Moreover, let $J$ be the set of reflections $s$ of $R$ that commute with all $r$ in $I$. If $\bv = \sum_{s \in J} v_s\alpha_s$, then $\bv^{b} = \bv$ for all $b \in W_I$. 
  \end{lemma}
  
  \begin{proof}
  	We have $(a,\bu)(a,\bu) = (a^2, \bu^a + \bu)$. If $(a,\bu)$ is an involution, then $a$ is an involution and $\bu^a + \bu =\bo$. Now $a$ is a product of elements of $I$. We have $\bv^{r} = \bv - \langle \bv, \alpha_r\rangle \alpha_r^{\vee}$ for all $\bu \in Z$ and $r \in R$. Hence, inductively, $\bu^a =  \bu - \mathbf{x}$ for some $\mathbf{x} \in \langle \alpha_r : r \in I \rangle$. So $\bu^a + \bu = 2\bu - \mathbf{x}$. For this to equal zero, clearly $\bu \in \langle \alpha_r : r \in I \rangle$. That is, $u_r = 0$ whenever $r \notin I$.  For the second part, observe that for any $r \in I$ and $s \in J$ we have $\langle \alpha_r, \alpha_s\rangle = 0$. Therefore $\bv^r = \bv$ for all $r \in I$. Hence $\bv^b = \bu$ for all $b \in W_I$.
  \end{proof}
  
  \begin{lemma}\label{4.3.5} Let $(a,\bu)$ be an involution in $\tilde W$, with $X= (a,\bu)^{\tilde W}$. Suppose $b$ commutes with $a$ in $W$, where $b\in a^W$. Then there is $(b,\bv )\in X$ such that $(a,\bu) \leftrightarrow (b,\bv)$.
  \end{lemma}
  \begin{proof}
  Since $b$ is conjugate in $W$ to $a$, there is $g
 \in W$ with $b=g^{-1}ag$.
 Let $\bw=\frac{1}{2}(\bu-\bu^g)$. Then
 \begin{align*}
 (a,\bu)^{(g,\bw)}&=
 (g^{-1}ag,{\bu}^g+\bw-\bw^{g^{-1}ag})
 =(b, {\bu}^g+\bw-\bw^b). 
   \end{align*} 
Set  $  \bv = {\bu}^g+\bw-\bw^b$. We claim that   $(a,\bu) \leftrightarrow (b,\bv)$. Certainly 
$(b,\bv)\in X$. Note that since $(a,\bu)$ is an involution, $\bu^a=-\bu$; also since $g^{-1}ag=b$ we have $gb=ag$ and so $\bu^{gb}=\bu^{ag}=-
 \bu^g$. Thus
\begin{align*}
  \bv& = {\bu}^g+\bw-\bw^b =\bu^g + \textstyle\frac{1}{2}(\bu-\bu^g)-
  \textstyle\frac{1}{2}(\bu^b-\bu^{g^b})\\
 &=\bu^g + \textstyle\frac{1}
 {2}\bu- \textstyle\frac{1}
  {2}\bu^g- \textstyle\frac{1}
   {2}\bu^b- \textstyle\frac{1}
    {2}\bu^g\\ & =\textstyle\frac{1}
       {2}(\bu-\bu^b).
 \end{align*} 
 Now
 \begin{align*} 
  (a,\bu)(b,\bv)&=(ab,\bu^b+\bv ) =(ab, \bu^b+\textstyle\frac{1}
         {2}(\bu-\bu^b)) =(ab, \textstyle\frac{1}
           {2}(\bu+\bu^b))\\  
           (b,\bv)(a,\bu)&=(ba,\bv^a+\bu)  
           =(ab,  \textstyle\frac{1}
            {2}\bu^a- \textstyle\frac{1}
           {2}\bu^{ba}+\bu)
        =(ab,  \textstyle\frac{1}
          {2}(\bu-\bu^{ab}))
       =(ab, \textstyle\frac{1}
                 {2}(\bu+\bu^b)).   \end{align*} 
                 Thus  $(a,\bu) \leftrightarrow (b,\bv)$, as required.
                 \end{proof}
  \begin{prop}\label{alternative2} Let $X$ be a conjugacy class of involutions in $\tilde W$ containing $(a, \bu)$. 
    Suppose there is an integer $k$ such that whenever  $(a, \bu') \in X$, we have $d((a,\bu'),(a, \bu)) \leq k$, and also that $\C(W,\hat X)$ is connected with diameter $d$. Then $\C(\tilde W, X)$ is connected with diameter at most $d + k$. 	
    \end{prop}
  	
  \begin{proof} Let $(b,\bv) \in X$.
  By hypothesis  $\C(W,\hat X)$ is connected with diameter $d$, meaning $d(a,b) \leq d$. Hence, by Lemma \ref{4.3.5}, there is a corresponding path 
   $(b,\bv)
    \leftrightarrow  \cdots \leftrightarrow (a,\bu')$, for appropriate $\bu'$, in $\C(\tilde W, X)$ of length at most $d$.
    By hypothesis $d((a,\bu' ), (a,\bu )) \leq k$. Therefore, $d((b,\bv), (a,\bu))\leq d+k$. \end{proof}

  \section{Types $\tilde F_4$ and $\tilde G_2$} \label{sec:f4tilde}
  
  Let $W$  be of type $F_4$, with associated root system $\Phi$ in $\rr^4$ and simple roots $\{\alpha_1, \alpha_2, \alpha_3, \alpha_4\}$. The root system $\Phi$ consists of the 24 long roots $\pm e_i \pm e_j$ ($1 \leq i<j \leq 4$) and 24 short roots: eight of the form $\pm e_i$ and sixteen of the form $\frac{1}{2}(\pm e_1 \pm e_2 \pm e_3 \pm e_4)$.  We may set $\alpha_1= \frac{1}{2}{(e_1-e_2-e_3-e_4)}, \alpha_2=e_4, \alpha_3=e_3-e_4$ and  $\alpha_4=e_2-e_3$. The highest root $\tilde \alpha$ is then $e_1 + e_2$. Then $W = \langle s_{\alpha_1}, s_{\alpha_2}, s_{\alpha_3}, s_{\alpha_4}\rangle$ and for the simple reflections of $\tilde W$ we can take $r_i = (s_{\alpha_i}, 0)$ for $i \in \{1, 2,3, 4\}$ and $r_5 = (s_{\tilde\alpha},\tilde\alpha^{\vee})$.  The Coxeter graph for $\tilde W$ is as shown in Figure \ref{fig:f4tilde}.  (The subgroup $\langle r_1, r_2, r_3, r_4\rangle$ is of course isomorphic to $W$ and we may for convenience identify it with $W$ on occasion.)

  \begin{figure}[h!]  \begin{center}
  \linethickness{0.4pt}\unitlength 0.8mm
  \begin{picture}(58.00,2)(5,10)
  \put(10,10.00){\line(1,0){12}}
  \put(22,10.80){\line(1,0){12}}
  \put(34,10.00){\line(1,0){24}}
  \put(22,9.2){\line(1,0){12}}
  \put(10,10.00){\circle*{2.00}}
  \put(22,10.00){\circle*{2.00}}
  \put(34,10.00){\circle*{2.00}}
  \put(46,10.00){\circle*{2.00}}
  \put(58,10.00){\circle*{2.00}}
  \put(10,6){\makebox(0,0)[cc]{$r_1$}}
  \put(22,6){\makebox(0,0)[cc]{$r_2$}}
  \put(34,6){\makebox(0,0)[cc]{$r_3$}}
  \put(46,6){\makebox(0,0)[cc]{$r_4$}}
  \put(58,6){\makebox(0,0)[cc]{$r_5$}}
  \end{picture}\end{center}
  \caption{Coxeter graph for type $\tilde F_4$} \label{fig:f4tilde}
  \end{figure}
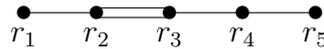
  
  Information about the commuting involution graphs of $W$ was obtained in \cite{Finite}. There are seven conjugacy classes of involutions in $W$. Let $a = (s_{\alpha_2}s_{\alpha_3})^2$. Then $\C(W,X)$ is connected with diameter 2. Apart from this, and the graph consisting of just the central involution, all the other commuting involution graphs are disconnected. See \cite{Finite} for more details.

    \renewcommand{\arraystretch}{1.1}
   \begin{table}[h!] \begin{center}
    	\begin{tabular}{c|c|c}
    		Graph & Representative $I$ & Underlying class in $W$ \\\hline
    		$A_1$ & $\{r_1\}$  & $A_1$  \\
    		$A_1$ & $\{r_3\}$  & $A_1$  \\
    		$A_1^2$ & $\{r_1, r_3\}$ & $A_1^2$ \\
    		$A_1^2 $ & $\{r_3, r_5\}$ & $B_2$  \\
    		$B_2$ & $\{ r_2, r_3\}$ & $B_2$  \\
    		$B_3 $ & $\{r_1, r_2, r_3\}$  & $B_3$ \\ 
    		$B_3 $ & $\{r_2, r_3, r_4\}$  & $B_3$  \\
    		$A_1^3$ & $\{r_1, r_3, r_5\}$ & $B_3$  \\
    		$B_2 \times A_1$ & $\{r_2,r_3,r_5\}$ & $B_3$  \\
    		$F_4$ & $\{r_1, r_2, r_3, r_4\}$ & $F_4$  \\
    		$B_4$ & $\{r_2, r_3, r_4, r_5\}$ & $F_4$\\
    		$B_3 \times A_1$ & $\{r_1, r_2, r_3, r_5\}$ & $F_4$  \\
    	\end{tabular} \caption{Conjugacy classes in $\tilde F_4$}\label{tab:conjf4tilde}\end{center}
    \end{table}
  
  Now let us consider the affine group $\tilde W$. By Theorem \ref{richardson}, every involution conjugacy class corresponds to a standard parabolic subgroup $W_I$ with a central involution $w_I$. We can identify the possibilities by finding subgraphs of the Coxeter graph for $\tilde F_4$ which correspond to Coxeter groups having nontrivial centres. Table \ref{tab:conjf4tilde} shows, for each subgraph giving rise to a distinct involution conjugacy class $X$, a representative $I \subseteq \{r_1, \ldots, r_5\}$ for which $w_I \in X$, along with the name, in the third column, of the Coxeter graph for the underlying class $\hat X$ in $W$. To determine the Coxeter graph corresponding to $\hat X$, note that if $I \subseteq W$, then $\hat X$ has the same graph as $X$. If $x = (a,\bu)$ is a product of $k$ reflections in $\tilde W$, then $a$ must be a product of $k$ reflections in $W$, so, for example, the conjugacy class of $\tilde W$ corresponding to the $B_4$ graph must have the class corresponding to $F_4$ as its underlying class in $W$. The only instance where this does not immediately tell us the type of the underlying class is the case of $A_1^2$ when $I = \{r_3, r_5\}$. Here, the underlying class might be type $B_2$ or type $A_1^2$. Note that $r_3 = s_{e_3 - e_4}$ and $r_5 = s_{e_1 + e_2,1}$. So the underlying class is the conjugacy class of $s_{e_3 - e_4}s_{e_1 + e_2}$ in $W$. One can check that $r_1r_4r_3r_2r_3r_4(e_1 + e_2) = e_3 + e_4$ and $r_1r_4r_3r_2r_3r_4(e_3 - e_4) = e_3 - e_4$. Hence $(s_{e_3 - e_4}s_{e_1 + e_2})^{r_4r_3r_2r_3r_4r_1} = s_{e_3 - e_4}s_{e_3 + e_4} = s_{e_3}s_{e_4} = r_3r_2r_3r_2 = (r_3r_2)^2$. Consequently the underlying conjugacy class in this case is of type $B_2$.

  \begin{thm}
  	\label{thmf4} Let $\tilde W$ be of type $\tilde F_4$, with graph shown in Figure \ref{fig:f4tilde}. If $X$ is the conjugacy class of $(r_2r_3)^2$ or $r_3r_5$ in $\tilde F_4$, then $\C(\tilde W, X)$ is connected with diameter at most 4. Otherwise, $\C(\tilde W, X)$ is disconnected.  
  \end{thm}

  \begin{proof} 
  	Let $X$ be a conjugacy class in $\tilde W$. If the underlying class in $W$ is anything other than type $B_2$ or $F_4$, then $\C(\tilde W, X)$ is disconnected, by Lemma \ref{obvious}. If the underlying class is type $F_4$ then $\C(\tilde W, X)$ is disconnected by Lemma \ref{centralinv}. So we are reduced to the case where the underlying class in $W$ is type $B_2$. There are two classes in $\tilde W$ where this happens, one containing $(r_2r_3)^2 = ((s_{\alpha_2}s_{\alpha_3})^2, \textbf{0})$, and one containing $r_3r_5 = ((s_{\alpha_2}s_{\alpha_3})^2, \tilde\alpha^\vee)$.
  	
  	Let $X$ be the conjugacy class of $(r_2r_3)^2$ in $\tilde W$; its underlying class $\hat X$ in $W$ is of type $B_2$. Note that $\hat X$ contains $((s_{\alpha_2}s_{\alpha_3})^2 = s_{e_3}s_{e_4}$ and thus also $s_{e_1}s_{e_2}$ (for example via the conjugating element $s_{e_1-e_3}s_{e_2 - e_4}$). In 
  	Proposition~\ref{alternative}, set $a = s_{e_1}s_{e_2}$ and $\hat x = s_{e_3}s_{e_4}$. Suppose $(a, \bu) \in X$. Now $(a,\bu)(\hat x, \bo) = (s_{e_1}s_{e_2}s_{e_3}s_{e_4}, \bv)$ for the appropriate $\bv$. This is clearly an involution because $s_{e_1}s_{e_2}s_{e_3}s_{e_4}$ acts as $-1$ on $Z$. Thus $(\hat x, \bo)$ commutes with $(a\,bu)$. Therefore we can apply Proposition~\ref{alternative} with $k=1$ and $d=2$, to see that $\C(\tilde W,X)$ is connected with diameter at most 3.\\
  	Now let $X$ be the conjugacy class of $r_3r_5$ in $ \tilde W$; again its underlying class in $W$ is of type $B_2$. Let $(a, \bu), (a, \bu') \in X$, where again $a = s_{e_3}s_{e_4}$. By Lemma \ref{4.3.5} there is some $\bv \in Z$ for which $(a, \bu) \leftrightarrow (s_{e_1}s_{e_2}, \bv)$. But now $(a, \bu')(s_{e_1}s_{e_2}, \bv) = (s_{e_1}s_{e_2}s_{e_3}s_{e_4}, \bv')$ for some $\bv'$, which is an involution. Thus in Proposition \ref{alternative2} we have $k=2$ and $d=2$, meaning that $\C(\tilde W, X)$ is connected with diameter at most 4.
  \end{proof}

    The Coxeter graph of type $\tilde G_2$ is as follows.

  \begin{center}  	\unitlength 1.00mm
    	\linethickness{0.4pt}
    	\begin{picture}(58.00,10)(-10,4)
    	
    	\put(10,9.00){\line(1,0){12}}
    	\put(10,9.00){\circle*{2.00}}
    	\put(22,9.00){\circle*{2.00}}
    	\put(34,9.00){\circle*{2.00}}
    	\put(16,12){\makebox(0,0)[cc]{$6$}}
    	\put(10,14){\makebox(0,0)[cc]{$r_{1}$}}
    	\put(22,14){\makebox(0,0)[cc]{$r_{2}$}}
    	\put(32,14){\makebox(0,0)[cc]{$r_{3}$}}
    	\put(-10,10){\makebox(0,0)[lc]{$\tilde G_2$}}  
    	\put(22,9.00){\line(1,0){12}}
    	\end{picture}\end{center}

 The subgraphs corresponding to parabolic subgroups $W_I$ for which $w_I$ is central  are of types $A_1$, $A_1
  	\times A_1$ or $G_2$.  It turns out that in all these cases, the commuting involution graphs are disconnected.
  	
  \begin{prop} 
  Let $X$ be a conjugacy class in the affine Coxeter group $\tilde W$ of type $\tilde{G_2}$. Then $C(\tilde W,X)$ is disconnected.
  \end{prop}
  \begin{proof} 
When $\tilde W$ is of type $\tilde G_2$, the underlying Weyl group $W$ is dihedral of order 12. It has three conjugacy classes of involutions: two classes of reflections each with three elements, and one class consisting of the unique central involution $\tilde w$. The commuting involution graphs for the classes of reflections in $W$ are completely disconnected. Now let $x \in X$. If $\hat x$ is a reflection, then $\graph$ is disconnected by Lemma~\ref{obvious}. Otherwise, $\hat x = \tilde w$, and so the disconnectedness of \graph\ follows from Proposition \ref{centralinv}.
  \end{proof}

 \end{document}